\documentclass[reqno,a4paper]{amsart}
\usepackage{amssymb}
\usepackage{amsmath}
\usepackage{amsfonts}

\newtheorem{theorem}{Theorem}[section]
\theoremstyle{plain}

\newtheorem{corollary}{Corollary}[section]

\newtheorem{definition}{Definition}[section]
\newtheorem{example}{Example}[section]

\newtheorem{lemma}{Lemma}[section]

\numberwithin{equation}{section}
\input{tcilatex}

\begin{document}
\title{On $f$-Eikonal Helices And $f$-Eikonal Slant Helices In Riemannian
Manifolds}
\author{Ali \c{S}enol}
\address{Department of Mathematics, Faculty of Science, \c{C}ank\i r\i\
Karatekin University, \c{C}ank\i r\i , Turkey}
\email{asenol@karatekin.edu.tr}
\urladdr{}
\author{Evren Z\i plar}
\address{Department of Mathematics, Faculty of Science, University of
Ankara, Tando\u{g}an, Turkey}
\email{evrenziplar@yahoo.com}
\author{Yusuf Yayl\i }
\address{Department of Mathematics, Faculty of Science, University of
Ankara, Tando\u{g}an, Turkey}
\email{yayli@science.ankara.edu.tr}
\urladdr{}
\date{November 21, 2012.}
\keywords{Eikonal helice, eikonal slant helice, harmonic curvature.\\
Corresponding author: Evren ZIPLAR, evrenziplar@yahoo.com}

\begin{abstract}
In this paper, we define $f$-eikonal helix curves and $f$-eikonal $V_{n}$%
-slant helix curves in a $n$-dimensional Riemannian manifold. Also, we give
the definition of harmonic curvature functions related to $f$-eikonal helix
curves and $f$-eikonal $V_{n}$-slant helix curves in a $n$-dimensional
Riemannian manifold. Moreover, we give characterizations for $f$-eikonal
helix curves and $f$-eikonal $V_{n}$-slant helix curves by making use of the
harmonic curvature functions.
\end{abstract}

\maketitle

\section{\textbf{Introduction}}

The helices share common origins in the geometries of the platonic solids,
with inherent hierarchical potential that is typical of biological
structures.The helices provide an energy-efficient solution to close-packing
in molecular biology, a common motif in protein construction, and a readily
observable pattern at many size levels throughout the body. The helices are
described in a variety of anatomical structures, suggesting their importance
to structural biology and manual therapy \cite{scarr}.

A general helix in Euclidean 3-space $E^{3}$ is defined by the property that
tangent makes a constant angle with a fixed straight line(the axis of
general helix). A classical result stated by Lancret in 1802 and first
proved by de Saint Venant in 1845 (\cite{Lancret} and \cite{Struik}) is: A
necessary and sufficient condition that a curve be a general helix is that
the ratio of curvature to torsion be constant. In \cite{haci}$,$ \"{O}zdamar
and Hac\i saliho\u{g}lu defined harmonic curvature functions $H_{i}$ $\left(
1\leq i\leq n-2\right) $ of a curve $\alpha $ in $n-$dimensional Euclidean
space $E^{n}$. They generalized inclined curves (general helix) in $E^{3}$
to $E^{n}$ and then gave a characterization for the inclined curves in $%
E^{n} $. Then, Izumiya and Takeuchi defined a new kind of helix (slant
helix) and they gave a characterization of slant helices in Euclidean $3-$%
space $E^{3}$ \cite{izumiya}. Kula and Yayl\i\ have studied spherical images
of tangent indicatrix and binormal indicatrix of a slant helix and they
showed that the spherical images are spherical helix \cite{K and Y}. In
2008, \"{O}nder \emph{et al}. defined a new kind of slant helix in Euclidean 
$4-$space $E^{4} $ which is called $B_{2}-$slant helix and they gave some
characterizations of these slant helices in Euclidean $4-$space $E^{4}$ \cite%
{onder} . And then in 2009, G\"{o}k \emph{et al}.defined a new kind of slant
helix in Euclidean $n-$space $E^{n}$, $n>3$, which they called $V_{n}-$slant
helix and they gave some characterizations of these slant helices in
Euclidean $n-$space \cite{gok}. On the other hand, Camc\i\ \emph{et al.}give
some characterizations for a non-degenerate curve to be a generalized helix
by using its harmonic curvatures \cite{cam}.

Let $M$ be a Riemannian manifold, where $\left\langle ,\right\rangle $ is
the metric. Let $f:M\rightarrow 
\mathbb{R}
$ be a function and let $\nabla f$ be its gradient, i.e., $%
df(X)=\left\langle \nabla f,X\right\rangle $. We say that $f$ is eikonal if
it satisfies: $\left\Vert \nabla f\right\Vert $ is constant \cite{scala} . $%
\nabla f$ is used many areas of science such as mathematical physics and
geometry. So, $\nabla f$ is very important subject. For example, the
Riemannian condition $\left\Vert \nabla f\right\Vert ^{2}=1$ (for
non-constant $f$ on connected $M$) is precisely the eikonal equation of
geometrical optics. Thus on a connected $M$, a non-constant real valued $f$
is Riemannian iff $f$ satisfies this eikonal equation. In the geometrical
optical interpretation, the level sets of $f$ are interpreted as wave
fronts. The characteristics of the eikonal equation (as a partial
differential equation), are then the solutions of the gradient flow equation
for $f$ (an ordinary differential equation), $x^{\prime }=\func{grad}f(x)$,
which are geodesics of $M$ orthogonal to the level sets of $f$, and which
are parametrized by arc length. These geodesics can be interpreted as light
rays orthogonal to the wave fronts \cite{Fischer} .

In this work, we introduced $f$-eikonal helices and $f$-eikonal slant
helices in a $n$-dimensional Riemannian manifold $M^{n}$. Moreover, in $M^{n}
$, we give the definition of harmonic curvature functions. Also, we give new
characterizations for $f$-eikonal helix curves and $f$-eikonal $V_{n}$-slant
helix curves by making use of the harmonic curvature functions.

\section{\textbf{Preliminaries}}

In this section, we give some basic definitions from differential geometry.

\begin{definition}
Let $\alpha =\alpha \left( s\right) $ be a smooth curve parametrized by its
arc length s in $n$-dimensional Riemannian manifold $M^{n}$. If there exist
orthonormal frame fields $\left\{ V_{1}=\alpha ^{\prime },...,V_{n}\right\} $
along $\alpha $ and positive functions $k_{1}\left( s\right)
,...,k_{n-1}\left( s\right) $ satisfying the following system of ordinary
equations%
\begin{equation}
\nabla _{\alpha ^{\prime }}V_{i}\left( s\right) =-k_{i-1}\left( s\right)
V_{i-1}\left( s\right) +k_{i}\left( s\right) V_{i+1}\left( s\right) ,\text{ }%
i=1,...,n,
\end{equation}%
where $V_{0}=V_{n+1}=0$ and\ $\ \nabla _{\alpha ^{\prime }}$\ denotes the
Riemannian connexion along $\alpha $, then the curve $\alpha $ is called a
Frenet curve of proper $n$. And, the equation (2.1) is called the Frenet
formula for the Frenet curve $\alpha $. The functions $k_{i}\left( s\right) $
$(i=1,...,n-1)$ and the orthonormal frame $\left\{ V_{1},...,V_{n}\right\} $
are called the curvatures and the Frenet frame of $\alpha $, respectively 
\cite{M and A}.
\end{definition}

\begin{definition}
Let $M$ be a Riemannian manifold, where $\left\langle ,\right\rangle $ is
the metric. Let $f:M\rightarrow 
\mathbb{R}
$ be a function and let $\nabla f$ be its gradient, i.e., $%
df(X)=\left\langle \nabla f,X\right\rangle $. We say that $f$ is eikonal if
it satisfies:%
\begin{equation*}
\left\Vert \nabla f\right\Vert =\text{constant.}
\end{equation*}%
\cite{scala}.
\end{definition}

\section{$f$\textbf{-eikonal helix curves and their harmonic curvature
functions}}

In this section, we define $f$-eikonal helix curves and we give
characterizations for a $f$-eikonal helix curve in $n$-dimensional
Riemannian manifold $M^{n}$ by using harmonic curvature functions of the
curve.

\begin{definition}
Let $\alpha \left( s\right) :I\subset 
\mathbb{R}
\rightarrow M^{n}$ be a Frenet curve of proper $n$ in $n$-dimensional
Riemannian manifold $M^{n}$ and let the functions $k_{i}\left( s\right) $ ($%
i=1,...,n-1$) be curvatures of the curve $\alpha $. Harmonic curvatures of
the curve $\alpha $ are defined by $H_{i}:I\subset 
\mathbb{R}
\rightarrow 
\mathbb{R}
$ along $\alpha $ in $M^{n}$ , $i=1,...,n-2$, such that%
\begin{equation*}
H_{1}=\frac{k_{1}}{k_{2}},\text{ }H_{i}=\left\{ V_{1}\left[ H_{i-1}\right]
+k_{i}H_{i-2}\right\} \frac{1}{k_{i+1}}
\end{equation*}%
for $i=2,...,n-2$, where $V_{1}$ is the unit vector tangent vector field of
the curve $\alpha $.
\end{definition}

\begin{definition}
\textbf{\ }Let $M^{n}$ be a Riemannian manifold with the metric $%
\left\langle ,\right\rangle $ and let $\alpha \left( s\right) $ be a Frenet
curve with the unit tangent vector field $V_{1}$ in $M^{n}$. Let $%
f:M^{n}\rightarrow 
\mathbb{R}
$ be a eikonal function along curve $\alpha $, i.e. $\left\Vert \nabla
f\right\Vert =$constant along the curve $\alpha $. If the function $%
\left\langle \nabla f,V_{1}\right\rangle $ is non-zero constant along $%
\alpha $, then $\alpha $ is called a $f$-eikonal helix curve. And, $\nabla f$
is called the axis of the $f$-eikonal helix curve $\alpha $.
\end{definition}

\begin{example}
We consider the Riemannian manifold $M^{3}=%
\mathbb{R}
^{3}$ with the Euclidean metric $\left\langle ,\right\rangle $. Let%
\begin{equation*}
f:M^{3}\rightarrow 
\mathbb{R}%
\end{equation*}%
\begin{equation*}
\left( x,y,z\right) \rightarrow f\left( x,y,z\right) =x^{2}+y+z^{2}
\end{equation*}%
be a function defined on $M^{3}$. Then, the curve%
\begin{equation*}
\alpha :I\subset 
\mathbb{R}
\rightarrow M^{3}
\end{equation*}%
\begin{equation*}
s\rightarrow \alpha \left( s\right) =(\cos \frac{s}{\sqrt{2}},\frac{s}{\sqrt{%
2}},\sin \frac{s}{\sqrt{2}})
\end{equation*}%
is a $f$-eikonal helix curve on $M^{3}$.

Firstly, we will show that $f$ is a eikonal function along the curve $\alpha 
$. If we compute $\nabla f$, we find $\nabla f$ as%
\begin{equation*}
\nabla f=\left( 2x,1,2z\right) \text{.}
\end{equation*}%
So, we get%
\begin{equation*}
\left\Vert \nabla f\right\Vert =\sqrt{4\left( x^{2}+z^{2}\right) +1}\text{.}
\end{equation*}%
And, if we compute $\left\Vert \nabla f\right\Vert $ along the curve $\alpha 
$, we find%
\begin{equation*}
\left\Vert \nabla f\right\Vert |_{\alpha }=\sqrt{5}=\text{constant.}
\end{equation*}%
That is, $f$ is a eikonal function along the curve $\alpha $.

Now, we will show that the function $\left\langle \nabla
f,V_{1}\right\rangle $ is non-zero constant along the curve $\alpha $. Since%
\begin{equation*}
\nabla f|_{\alpha }=(2\cos \frac{s}{\sqrt{2}},1,2\sin \frac{s}{\sqrt{2}})
\end{equation*}%
and%
\begin{equation*}
V_{1}=(-\frac{1}{\sqrt{2}}\sin \frac{s}{\sqrt{2}},\frac{1}{\sqrt{2}},\frac{1%
}{\sqrt{2}}\cos \frac{s}{\sqrt{2}})\text{ ,}
\end{equation*}%
we obtain 
\begin{equation*}
\left\langle \nabla f|_{\alpha },V_{1}\right\rangle =\frac{1}{\sqrt{2}}=%
\text{constant}
\end{equation*}%
along the curve $\alpha $. Consequently, $\alpha $ is a $f$-eikonal helix
curve on $M$.
\end{example}

\begin{theorem}
Let $M^{n}$ be a $n$-dimensional Riemannian manifold with the metric $%
\left\langle ,\right\rangle $ and complete connected smooth without
boundary. Let $M^{n}$ be isometric to a Riemannian product $N\times 
\mathbb{R}
$. Let us assume that $f:M^{n}\rightarrow 
\mathbb{R}
$ be a non-trivial affine function (see main Theorem in \cite{Innami}) and $%
\alpha \left( s\right) $ be a Frenet curve of proper $n$ in $M^{n}$. If $%
\alpha $ is a $f$-eikonal helix curve with the axis $\nabla f$, then the
system%
\begin{equation}
\left\langle V_{i+2},\nabla f\right\rangle =H_{i}\left\langle V_{1},\nabla
f\right\rangle ,\text{ }i=1,...,n-2
\end{equation}%
holds, where $\left\{ V_{1},...,V_{n}\right\} $ and $\left\{
H_{1},...,H_{n-2}\right\} $ are the Frenet frame and the Harmonic curvatures
of $\alpha $, respectively.
\end{theorem}

\begin{proof}
Since $\left\{ V_{1},...,V_{n}\right\} $ is the orthonormal frame of the
curve $\alpha $ in $M^{n}$, $\nabla f$ can be expressed in the form%
\begin{equation}
\nabla f=\lambda _{1}V_{1}+...+\lambda _{n}V_{n}\text{.}
\end{equation}%
Doing dot product with $V_{1}$ in each part of (3.2), we get%
\begin{equation}
\left\langle \nabla f,V_{1}\right\rangle =\lambda _{1}=\left\Vert \nabla
f\right\Vert \cos \left( \theta \right) =\text{constant}
\end{equation}%
since $\alpha $ is a $f$-eikonal helix curve. If we take the derivative in
each part of (3.3) in the direction $V_{1}$ in $M^{n}$, then we have%
\begin{equation}
\left\langle \nabla _{V_{1}}\nabla f,V_{1}\right\rangle +\left\langle \nabla
f,\nabla _{V_{1}}V_{1}\right\rangle =0\text{.}
\end{equation}%
On the other hand, from Lemma 2.3 (see \cite{Sakai}), $\nabla f$ is parallel
in $M^{n}$. That is, $\nabla _{V_{1}}\nabla f=0$. Hence, by using (3.4) and
Frenet formulas, we obtain%
\begin{equation}
k_{1}\left\langle \nabla f,V_{2}\right\rangle =0\text{.}
\end{equation}%
And, since $k_{1}$ is positive function, from (3.5), we get%
\begin{equation}
\left\langle \nabla f,V_{2}\right\rangle =0\text{.}
\end{equation}%
By taking the derivative in each part of (3.6) in the direction $V_{1}$ in $%
M^{n}$, we can write the equality%
\begin{equation}
\left\langle \nabla _{V_{1}}\nabla f,V_{2}\right\rangle +\left\langle \nabla
f,\nabla _{V_{1}}V_{2}\right\rangle =0\text{.}
\end{equation}%
And, since $\nabla _{V_{1}}\nabla f=0$, by using (3.7) and Frenet formulas,
we obtain%
\begin{equation}
-k_{1}\left\langle \nabla f,V_{1}\right\rangle +k_{2}\left\langle \nabla
f,V_{3}\right\rangle =0\text{.}
\end{equation}%
Therefore, from (3.8), we have 
\begin{equation}
\left\langle \nabla f,V_{3}\right\rangle =\lambda _{3}=\frac{k_{1}}{k_{2}}%
\left\langle \nabla f,V_{1}\right\rangle \text{.}
\end{equation}%
Moreover, since $H_{1}=\dfrac{k_{1}}{k_{2}}$, from (3.9), we can write%
\begin{equation*}
\left\langle \nabla f,V_{3}\right\rangle =H_{1}\left\langle \nabla
f,V_{1}\right\rangle \text{.}
\end{equation*}%
It follows that the equality (3.1) is true for $i=1$. According to the
induction theory, let us assume that the equality (3.1) is true for all $k$,
where $1\leq k\leq i$ for some positive integers $i$. Then, we will prove
that the equality (3.1) is true for $i+1$. Since the equality (3.1) is true
for some positive integers $i$, we can write%
\begin{equation}
\left\langle V_{i+2},\nabla f\right\rangle =H_{i}\left\langle V_{1},\nabla
f\right\rangle
\end{equation}%
for some positive integers $i$. If we take derivative in each part of (3.10)
in the direction $V_{1}$ in $M^{n}$, we have%
\begin{equation}
\left\langle \nabla _{V_{1}}V_{i+2},\nabla f\right\rangle +\left\langle
V_{i+2},\nabla _{V_{1}}\nabla f\right\rangle =V_{1}\left[ H_{i}\left\langle
V_{1},\nabla f\right\rangle \right] \text{.}
\end{equation}%
And, by using (3.11) and Frenet formulas, we get%
\begin{equation}
-k_{i+1}\left\langle V_{i+1},\nabla f\right\rangle +k_{i+2}\left\langle
V_{i+3},\nabla f\right\rangle +\left\langle V_{i+2},\nabla _{V_{1}}\nabla
f\right\rangle =V_{1}\left[ H_{i}\left\langle V_{1},\nabla f\right\rangle %
\right] \text{.}
\end{equation}%
Morever, $\nabla _{V_{1}}\nabla f=0$. Hence, from (3.12), we can write%
\begin{equation}
-k_{i+1}\left\langle V_{i+1},\nabla f\right\rangle +k_{i+2}\left\langle
V_{i+3},\nabla f\right\rangle =V_{1}\left[ H_{i}\left\langle V_{1},\nabla
f\right\rangle \right] \text{.}
\end{equation}%
\ And so, we obtain%
\begin{equation}
\left\langle V_{i+3},\nabla f\right\rangle =\left\{ V_{1}\left[
H_{i}\left\langle V_{1},\nabla f\right\rangle \right] +k_{i+1}\left\langle
V_{i+1},\nabla f\right\rangle \right\} \frac{1}{k_{i+2}}\text{.}
\end{equation}%
On the other hand, since the equality (3.1) is true for $i-1$ according to
the induction hypothesis, we have%
\begin{equation}
\left\langle V_{i+1},\nabla f\right\rangle =H_{i-1}\left\langle V_{1},\nabla
f\right\rangle \text{.}
\end{equation}%
Therefore, by using (3.14) and (3.15), we get%
\begin{equation}
\left\langle V_{i+3},\nabla f\right\rangle =\left\{ V_{1}\left[ H_{i}\right]
+k_{i+1}H_{i-1}\right\} \cdot \frac{1}{k_{i+2}}\cdot \text{ }\left\langle
V_{1},\nabla f\right\rangle \text{.}
\end{equation}%
Moreover, we obtain%
\begin{equation}
\text{ }H_{i+1}=\left\{ V_{1}\left[ H_{i}\right] +k_{i+1}H_{i-1}\right\} 
\frac{1}{k_{i+2}}\text{.}
\end{equation}%
for $i+1$ in the Definition 3.1. So, we have%
\begin{equation*}
\left\langle V_{i+3},\nabla f\right\rangle =H_{i+1}\left\langle V_{1},\nabla
f\right\rangle
\end{equation*}%
by using (3.16) and (3.17). It follows that the equality (3.1) is true for $%
i+1$. Consequently, we get%
\begin{equation*}
\left\langle V_{i+2},\nabla f\right\rangle =H_{i}\left\langle V_{1},\nabla
f\right\rangle
\end{equation*}%
for all $i$ according to induction theory. This completes the proof. \ \ \ \
\ \ \ \ \ \ \ \ \ \ \ \ \ \ \ \ \ \ \ \ \ \ \ \ \ \ \ \ \ \ \ \ \ \ \ \ \ \
\ \ \ \ \ \ \ \ \ \ \ \ \ \ \ \ \ \ \ \ \ \ \ \ \ \ \ \ \ \ \ \ \ \ \ \ \ \
\ \ \ \ \ \ \ \ \ \ \ \ \ \ \ \ \ \ \ \ \ \ \ \ \ \ \ \ \ \ \ \ \ \ \ \ \ \
\ \ \ \ \ \ \ \ \ \ \ \ \ \ \ \ \ \ \ \ \ \ \ \ \ \ \ \ \ \ \ \ \ \ \ \ \ \
\ \ \ \ \ \ 
\end{proof}

\begin{theorem}
Let $M^{n}$ be a $n$-dimensional Riemannian manifold with the metric $%
\left\langle ,\right\rangle $ and complete connected smooth without
boundary. Let $M^{n}$ be isometric to a Riemannian product $N\times 
\mathbb{R}
$. Let us assume that $f:M^{n}\rightarrow 
\mathbb{R}
$ be a non-trivial affine function (see main Theorem in \cite{Innami}) and $%
\alpha \left( s\right) $ be a Frenet curve of proper $n$ in $M^{n}$. If $%
\alpha $ is a $f$-eikonal helix curve with the axis $\nabla f$, then the
axis of the curve $\alpha $:%
\begin{equation*}
\nabla f=\left\Vert \nabla f\right\Vert \cos \left( \theta \right) \left(
V_{1}+H_{1}V_{3}+...+H_{n-2}V_{n}\right) \text{,}
\end{equation*}%
where $\left\{ V_{1},...,V_{n}\right\} $ and $\left\{
H_{1},...,H_{n-2}\right\} $ are the Frenet frame and the Harmonic curvatures
of $\alpha $, respectively.
\end{theorem}

\begin{proof}
Since $\alpha $ is a $f$-eikonal helix curve , we can write%
\begin{equation}
\left\langle \nabla f,V_{1}\right\rangle =\text{constant.}
\end{equation}%
If we take the derivative in each part of (3.18) in the direction $V_{1}$ in 
$M^{n}$, then we have%
\begin{equation}
\left\langle \nabla _{V_{1}}\nabla f,V_{1}\right\rangle +\left\langle \nabla
f,\nabla _{V_{1}}V_{1}\right\rangle =0\text{.}
\end{equation}%
On the other hand, from Lemma 2.3 (see \cite{Sakai}), $\nabla f$ is parallel
in $M^{n}$. That's why, $\nabla _{V_{1}}\nabla f=0$. Then, we obtain%
\begin{equation}
k_{1}\left\langle \nabla f,V_{2}\right\rangle =0\text{.}
\end{equation}%
by using (3.19) and Frenet formulas. Since $k_{1}$ is positive function,
(3.20) implies that%
\begin{equation*}
\left\langle \nabla f,V_{2}\right\rangle =0\text{.}
\end{equation*}%
Hence, we can write the axis of $\alpha $ as%
\begin{equation}
\nabla f=\lambda _{1}V_{1}+\lambda _{3}V_{3}+...+\lambda _{n}V_{n}\text{.}
\end{equation}%
Moreover, from (3.21), we get%
\begin{eqnarray*}
\lambda _{1} &=&\left\langle \nabla f,V_{1}\right\rangle \\
\lambda _{3} &=&\left\langle \nabla f,V_{3}\right\rangle \\
&&. \\
&&. \\
&&. \\
\lambda _{n} &=&\left\langle \nabla f,V_{n}\right\rangle
\end{eqnarray*}%
by using dot product. On the other hand, from Theorem 3.1, we know that%
\begin{eqnarray}
\lambda _{1} &=&\left\langle \nabla f,V_{1}\right\rangle =\left\Vert \nabla
f\right\Vert \cos \left( \theta \right) \\
\lambda _{3} &=&\left\langle \nabla f,V_{3}\right\rangle =\left\Vert \nabla
f\right\Vert \cos \left( \theta \right) H_{1}  \notag \\
&&.  \notag \\
&&.  \notag \\
&&.  \notag \\
\lambda _{n} &=&\left\langle \nabla f,V_{n}\right\rangle =\left\Vert \nabla
f\right\Vert \cos \left( \theta \right) H_{n-2}\text{.}  \notag
\end{eqnarray}%
Thus, it can be easily obtained the axis of the curve $\alpha $ as%
\begin{equation*}
\nabla f=\left\Vert \nabla f\right\Vert \cos \left( \theta \right) \left(
V_{1}+H_{1}V_{3}+...+H_{n-2}V_{n}\right)
\end{equation*}%
by making use of \ the equality (3.21) and the system (3.22). This completes
the proof.
\end{proof}

\begin{theorem}
Let $M^{n}$ be a $n$-dimensional Riemannian manifold with the metric $%
\left\langle ,\right\rangle $ and complete connected smooth without
boundary. Let $M^{n}$ be isometric to a Riemannian product $N\times 
\mathbb{R}
$. Let us assume that $f:M^{n}\rightarrow 
\mathbb{R}
$ be a non-trivial affine function (see main Theorem in \cite{Innami} ) and $%
\alpha \left( s\right) $ be a Frenet curve of proper $n$ in $M^{n}$. If $%
\alpha $ is a $f$-eikonal helix curve, then $H_{n-2}\neq 0$ and $%
H_{1}^{2}+H_{2}^{2}+...+H_{n-2}^{2}$ is nonzero constant, where $\left\{
H_{1},...,H_{n-2}\right\} $ is the Harmonic curvatures of $\alpha $.
\end{theorem}

\begin{proof}
Let $\alpha $ be a $f$-eikonal helix curve and $\left\{
V_{1},...,V_{n}\right\} $ be the Frenet frame of $\alpha $. Then, from
Theorem 3.2, we know that%
\begin{equation}
\nabla f=\left\Vert \nabla f\right\Vert \cos \left( \theta \right) \left(
V_{1}+H_{1}V_{3}+...+H_{n-2}V_{n}\right) \text{.}
\end{equation}%
Therefore, from (3.23), we can write%
\begin{equation}
\left\langle \nabla f,\nabla f\right\rangle =\left\Vert \nabla f\right\Vert
^{2}\left( \cos ^{2}(\theta )+\sum\limits_{i=1}^{n-2}H_{i}^{2}\cos
^{2}(\theta )\right) \text{.}
\end{equation}%
Moreover, by the definition of Riemannian metric, we have 
\begin{equation}
\left\langle \nabla f,\nabla f\right\rangle =\left\Vert \nabla f\right\Vert
^{2}\text{.}
\end{equation}%
Hence, from (3.24) and (3.25), we obtain%
\begin{equation}
\cos ^{2}(\theta )+\sum\limits_{i=1}^{n-2}H_{i}^{2}\cos ^{2}(\theta )=1\text{%
.}
\end{equation}%
It follows that 
\begin{equation*}
H_{1}^{2}+H_{2}^{2}+...+H_{n-2}^{2}=\tan ^{2}(\theta )=\text{constant.}
\end{equation*}%
Now, we will show that $H_{n-2}\neq 0$. We assume that $H_{n-2}=0$. Then,
for $i=n-2$ in Theorem 3.1, 
\begin{equation}
\left\langle V_{n},\nabla f\right\rangle =H_{n-2}\left\langle V_{1},\nabla
f\right\rangle =0\text{.}
\end{equation}%
If we take derivative in each part of (3.27) in the direction $V_{1}$ on $%
M^{n}$, then we have%
\begin{equation}
\left\langle \nabla _{V_{1}}V_{n},\nabla f\right\rangle +\left\langle
V_{n},\nabla _{V_{1}}\nabla f\right\rangle =0\text{.}
\end{equation}%
On the other hand, from Lemma 2.3 (see \cite{Sakai} ), $\nabla f$ is
parallel in $M^{n}$. That's why, $\nabla _{V_{1}}\nabla f=0$. Hence, we get%
\begin{equation}
-k_{n-1}\left\langle V_{n-1},\nabla f\right\rangle =0
\end{equation}%
by using (3.28) and Frenet formulas. So, from (3.29), we deduce that $%
\left\langle V_{n-1},\nabla f\right\rangle =0$ since $k_{n-1}$ is positive.
For $i=n-3$ in Theorem 3.1, 
\begin{equation*}
\left\langle V_{n-1},\nabla f\right\rangle =H_{n-3}\left\langle V_{1},\nabla
f\right\rangle \text{.}
\end{equation*}%
And, since $\left\langle V_{n-1},\nabla f\right\rangle =0$, $H_{n-3}=0$.
Continuing this process, we get $H_{1}=0$. Let us recall that $H_{1}=\dfrac{%
k_{1}}{k_{2}}$, thus we have a contradiction. Because, all the curvatures
are nowhere zero. As a result, $H_{n-2}\neq 0$. This completes the proof.
\end{proof}

\begin{lemma}
Let $\alpha \left( s\right) $ be a Frenet curve of proper $n$ in $n$%
-dimensional Riemannian manifold $M^{n}$ and let $H_{n-2}\neq 0$ be for $%
i=n-2$. Then, $H_{1}^{2}+H_{2}^{2}+...+H_{n-2}^{2}$ is a nonzero constant if
and only if $V_{1}\left[ H_{n-2}\right] =H_{n-2}^{\prime }=-k_{n-1}H_{n-3}$,
where $V_{1}$ and $\left\{ H_{1},...,H_{n-2}\right\} $ are the unit vector
tangent vector field and the Harmonic curvatures of $\alpha $, respectively.
\end{lemma}

\begin{proof}
First, we assume that $H_{1}^{2}+H_{2}^{2}+...+H_{n-2}^{2}$ is a nonzero
constant . Consider the functions%
\begin{equation*}
H_{i}=\left\{ H_{i-1}^{\prime }+k_{i}H_{i-2}\right\} \frac{1}{k_{i+1}}
\end{equation*}%
for $3\leq i\leq n-2$. (see Definition 3.1). So, from the equality, we can
write%
\begin{equation}
k_{i+1}H_{i}=H_{i-1}^{\prime }+k_{i}H_{i-2}\text{, }3\leq i\leq n-2\text{.}
\end{equation}%
Hence, in (3.30), if we take $i+1$ instead of $i$, we get%
\begin{equation}
H_{i}^{\prime }=k_{i+2}H_{i+1}-k_{i+1}H_{i-1}\text{, }2\leq i\leq n-3
\end{equation}%
together with%
\begin{equation}
H_{1}^{\prime }=k_{3}H_{2}\text{.}
\end{equation}%
On the other hand, since $H_{1}^{2}+H_{2}^{2}+...+H_{n-2}^{2}$ is constant,
we have%
\begin{equation*}
H_{1}H_{1}^{\prime }+H_{2}H_{2}^{\prime }+...+H_{n-2}H_{n-2}^{\prime }=0
\end{equation*}%
and so,%
\begin{equation}
H_{n-2}H_{n-2}^{\prime }=-H_{1}H_{1}^{\prime }-H_{2}H_{2}^{\prime
}-...-H_{n-3}H_{n-3}^{\prime }\text{.}
\end{equation}%
By using (3.31) and (3.32), we obtain%
\begin{equation}
H_{1}H_{1}^{\prime }=k_{3}H_{1}H_{2}
\end{equation}%
and%
\begin{equation}
H_{i}H_{i}^{\prime }=k_{i+2}H_{i}H_{i+1}-k_{i+1}H_{i-1}H_{i}\text{, }2\leq
i\leq n-3\text{.}
\end{equation}%
Therefore, by using (3.33), (3.34) and (3.35), a algebraic calculus shows
that%
\begin{equation*}
H_{n-2}H_{n-2}^{\prime }=-k_{n-1}H_{n-3}H_{n-2}\text{.}
\end{equation*}%
Since $H_{n-2}\neq 0$, we get the relation $H_{n-2}^{\prime
}=-k_{n-1}H_{n-3} $.

Conversely, we assume that%
\begin{equation}
H_{n-2}^{\prime }=-k_{n-1}H_{n-3}
\end{equation}%
By using (3.36) and $H_{n-2}\neq 0$, we can write%
\begin{equation*}
H_{n-2}H_{n-2}^{\prime }=-k_{n-1}H_{n-2}H_{n-3}
\end{equation*}%
From (3.35), we have 
\begin{eqnarray*}
\text{for }i &=&n-3\text{, \ \ \ \ \ }H_{n-3}H_{n-3}^{\prime
}=k_{n-1}H_{n-3}H_{n-2}-k_{n-2}H_{n-4}H_{n-3} \\
\text{for }i &=&n-4\text{, \ \ \ \ \ }H_{n-4}H_{n-4}^{\prime
}=k_{n-2}H_{n-4}H_{n-3}-k_{n-3}H_{n-5}H_{n-4} \\
\text{for }i &=&n-5\text{, \ \ \ \ \ }H_{n-5}H_{n-5}^{\prime
}=k_{n-3}H_{n-5}H_{n-4}-k_{n-4}H_{n-6}H_{n-5} \\
&&\cdot \\
&&\cdot \\
&&\cdot \\
\text{for }i &=&2\text{, \ \ \ \ \ \ \ \ \ \ }H_{2}H_{2}^{\prime
}=k_{4}H_{2}H_{3}-k_{3}H_{1}H_{2}
\end{eqnarray*}%
and from (3.34), we have%
\begin{equation*}
H_{1}H_{1}^{\prime }=k_{3}H_{1}H_{2}\text{.}
\end{equation*}%
So, an algebraic calculus show that%
\begin{equation}
H_{1}H_{1}^{\prime }+H_{2}H_{2}^{\prime }+...+\text{\ }H_{n-5}H_{n-5}^{%
\prime }+H_{n-4}H_{n-4}^{\prime }+H_{n-3}H_{n-3}^{\prime
}+H_{n-2}H_{n-2}^{\prime }=0\text{.}
\end{equation}%
And, by integrating (3.37), we can easily say that%
\begin{equation*}
H_{1}^{2}+H_{2}^{2}+...+H_{n-2}^{2}
\end{equation*}%
is a non-zero constant. This completes the proof.
\end{proof}

\begin{corollary}
Let $M^{n}$ be a $n$-dimensional Riemannian manifold with the metric $%
\left\langle ,\right\rangle $ and complete connected smooth without
boundary. Let $M^{n}$ be isometric to a Riemannian product $N\times 
\mathbb{R}
$. Let us assume that $f:M^{n}\rightarrow 
\mathbb{R}
$ be a non-trivial affine function (see main Theorem in \cite{Innami} ) and $%
\alpha \left( s\right) $ be a Frenet curve of proper $n$ in $M^{n}$. If $%
\alpha $ is a $f$-eikonal helix curve, then $H_{n-2}^{\prime
}=-k_{n-1}H_{n-3}$.
\end{corollary}

\begin{proof}
It is obvious by using Theorem 3.3 and Lemma 3.1.
\end{proof}

\section{$f$\textbf{-eikonal }$V_{n}$\textbf{-slant helix curves and their
harmonic curvature functions}}

In this section, we define $f$-eikonal $V_{n}$-slant helix curves and we
give characterizations for a $f$-eikonal $V_{n}$-slant helix curve in $n$%
-dimensional Riemannian manifold $M^{n}$ by using harmonic curvature
functions in terms of $V_{n}$ of the curve.

\begin{definition}
\textbf{\ }Let $M^{n}$ be a Riemannian manifold and let $\alpha \left(
s\right) $ be a Frenet curve with the curvatures $k_{i}$. Then, harmonic
curvature functions of $\alpha $ are defined by $H_{i}^{\ast }:I\subset 
\mathbb{R}
\rightarrow 
\mathbb{R}
$ along $\alpha $ in $M^{n}$,%
\begin{equation*}
H_{0}^{\ast }=0,H_{1}^{\ast }=\frac{k_{n-1}}{k_{n-2}},H_{i}^{\ast }=\left\{
k_{n-i}H_{i-2}^{\ast }-H_{i-1}^{\ast \prime }\right\} \frac{1}{k_{n-(i+1)}}
\end{equation*}%
for $2\leq i\leq n-2$.
\end{definition}

\begin{definition}
\textbf{\ }Let $M^{n}$ be a Riemannian manifold with the metric $%
\left\langle ,\right\rangle $ and let $\alpha \left( s\right) $ be a Frenet
curve with the orthonormal frame $\left\{ V_{1},...,V_{n}\right\} $ in $M^{n}
$. Let $f:M^{n}\rightarrow 
\mathbb{R}
$ be a eikonal function along curve $\alpha $, i.e. $\left\Vert \nabla
f\right\Vert =$constant along the curve $\alpha $. If the function $%
\left\langle \nabla f,V_{n}\right\rangle $ is non-zero constant along $%
\alpha $, then $\alpha $ is called a $f$-eikonal $V_{n}$-slant helix curve.
And, $\nabla f$ is called the axis of the $f$-eikonal $V_{n}$-slant helix
curve $\alpha $.
\end{definition}

\begin{theorem}
Let $M^{n}$ be a $n$-dimensional Riemannian manifold with the metric $%
\left\langle ,\right\rangle $ and complete connected smooth without
boundary. Let $M^{n}$ be isometric to a Riemannian product $N\times 
\mathbb{R}
$. Let us assume that $f:M^{n}\rightarrow 
\mathbb{R}
$ be a non-trivial affine function (see main Theorem in \cite{Innami}) and $%
\alpha \left( s\right) $ be a Frenet curve of proper $n$ in $M^{n}$. If $%
\alpha $ is a $f$-eikonal $V_{n}$-slant helix curve with the axis $\nabla f$%
, then the system%
\begin{equation}
\left\langle V_{n-(i+1)},\nabla f\right\rangle =H_{i}^{\ast }\left\langle
V_{n},\nabla f\right\rangle \text{, }i=1,..,n-2
\end{equation}%
holds, where $\left\{ V_{1},V_{2},...,V_{n}\right\} $ and $\left\{
H_{1}^{\ast },...,H_{n-2}^{\ast }\right\} $ are the Frenet frame and the
harmonic curvatures of $\alpha $, respectively.
\end{theorem}

\begin{proof}
Since $\left\{ V_{1},...,V_{n}\right\} $ is the orthonormal frame of the
curve $\alpha $ in $M^{n}$, $\nabla f$ can be expressed in the form%
\begin{equation}
\nabla f=\lambda _{1}V_{1}+...+\lambda _{n}V_{n}\text{.}
\end{equation}%
Doing dot product with $V_{n}$ in each part of (4.2), we get%
\begin{equation}
\left\langle \nabla f,V_{n}\right\rangle =\lambda _{n}=\text{constant}
\end{equation}%
since $\alpha $ is a $f$-eikonal $V_{n}$-slant helix curve. If we take the
derivative in each part of (4.3) in the direction $V_{1}$ in $M^{n}$, then
we have%
\begin{equation}
\left\langle \nabla _{V_{1}}\nabla f,V_{n}\right\rangle +\left\langle \nabla
f,\nabla _{V_{1}}V_{n}\right\rangle =0\text{.}
\end{equation}%
On the other hand, from Lemma 2.3 (see \cite{Sakai}), $\nabla f$ is parallel
in $M^{n}$. That is, $\nabla _{V_{1}}\nabla f=0$. Hence, by using (4.4) and
Frenet formulas, we obtain%
\begin{equation}
-k_{n-1}\left\langle \nabla f,V_{n-1}\right\rangle =0\text{.}
\end{equation}%
And, since $k_{n-1}$ is positive function, from (4.5), we get%
\begin{equation}
\left\langle \nabla f,V_{n-1}\right\rangle =0\text{.}
\end{equation}%
By taking the derivative in each part of (4.6) in the direction $V_{1}$ in $%
M^{n}$, we can write the equality%
\begin{equation}
\left\langle \nabla _{V_{1}}\nabla f,V_{n-1}\right\rangle +\left\langle
\nabla f,\nabla _{V_{1}}V_{n-1}\right\rangle =0\text{.}
\end{equation}%
And, since $\nabla _{V_{1}}\nabla f=0$, by using (4.7) and Frenet formulas,
we obtain%
\begin{equation}
-k_{n-2}\left\langle \nabla f,V_{n-2}\right\rangle +k_{n-1}\left\langle
\nabla f,V_{n}\right\rangle =0\text{.}
\end{equation}%
Therefore, from (4.8), we have%
\begin{equation}
\left\langle \nabla f,V_{n-2}\right\rangle =\frac{k_{n-1}}{k_{n-2}}%
\left\langle \nabla f,V_{n}\right\rangle \text{.}
\end{equation}%
Moreover, since $H_{1}^{\ast }=\dfrac{k_{n-1}}{k_{n-2}}$, from (4.9), we can
write%
\begin{equation*}
\left\langle \nabla f,V_{n-2}\right\rangle =H_{1}^{\ast }\left\langle \nabla
f,V_{n}\right\rangle \text{.}
\end{equation*}%
It follows that the equality (4.1) is true for $i=1$. According to the
induction theory, let us assume that the equality (4.1) is true for all $k$,
where $1\leq k\leq i$ for some positive integers $i$. Then, we will prove
that the equality (4.1) is true for $i+1$. Since the equality (4.1) is true
for some positive integers $i$, we can write%
\begin{equation}
\left\langle V_{n-(i+1)},\nabla f\right\rangle =H_{i}^{\ast }\left\langle
V_{n},\nabla f\right\rangle
\end{equation}%
for some positive integers $i$. If we take derivative in each part of (4.10)
in the direction $V_{1}$ in $M^{n}$, we have%
\begin{equation}
\left\langle \nabla _{V_{1}}V_{n-(i+1)},\nabla f\right\rangle +\left\langle
V_{n-(i+1)},\nabla _{V_{1}}\nabla f\right\rangle =V_{1}\left[ H_{i}^{\ast
}\left\langle V_{n},\nabla f\right\rangle \right] \text{.}
\end{equation}%
And, by using (4.11) and Frenet formulas, we get%
\begin{equation}
-k_{n-(i+2)}\left\langle V_{n-(i+2)},\nabla f\right\rangle
+k_{n-(i+1)}\left\langle V_{n-i},\nabla f\right\rangle +\left\langle
V_{n-(i+1)},\nabla _{V_{1}}\nabla f\right\rangle =V_{1}\left[ H_{i}^{\ast
}\left\langle V_{n},\nabla f\right\rangle \right] \text{.}
\end{equation}%
Morever, $\nabla _{V_{1}}\nabla f=0$. Hence, from (4.12), we can write%
\begin{equation}
-k_{n-(i+2)}\left\langle V_{n-(i+2)},\nabla f\right\rangle
+k_{n-(i+1)}\left\langle V_{n-i},\nabla f\right\rangle =V_{1}\left[
H_{i}^{\ast }\left\langle V_{n},\nabla f\right\rangle \right] \text{.}
\end{equation}%
And, from (4.13), we obtain%
\begin{equation}
\left\langle V_{n-(i+2)},\nabla f\right\rangle =\left\{ -V_{1}\left[
H_{i}^{\ast }\left\langle V_{n},\nabla f\right\rangle \right]
+k_{n-(i+1)}\left\langle V_{n-i},\nabla f\right\rangle \right\} \frac{1}{%
k_{n-(i+2)}}\text{.}
\end{equation}%
On the other hand, since the equality (4.1) is true for $i-1$ according to
the induction hypothesis, we have%
\begin{equation}
\left\langle V_{n-i},\nabla f\right\rangle =H_{i-1}^{\ast }\left\langle
V_{n},\nabla f\right\rangle \text{.}
\end{equation}%
Therefore, by using (4.14) and (4.15), we get 
\begin{eqnarray}
\left\langle V_{n-(i+2)},\nabla f\right\rangle &=&\left\{ -V_{1}\left[
H_{i}^{\ast }\right] +k_{n-(i+1)}H_{i-1}^{\ast }\right\} \frac{1}{k_{n-(i+2)}%
}\left\langle V_{n},\nabla f\right\rangle \\
&=&\left\{ -H_{i}^{\ast \prime }+k_{n-(i+1)}H_{i-1}^{\ast }\right\} \frac{1}{%
k_{n-(i+2)}}\left\langle V_{n},\nabla f\right\rangle \text{.}  \notag
\end{eqnarray}%
Moreover, we obtain%
\begin{equation}
H_{i+1}^{\ast }=\left\{ k_{n-(i+1)}H_{i-1}^{\ast }-H_{i}^{\ast \prime
}\right\} \frac{1}{k_{n-(i+2)}}
\end{equation}%
for $i+1$ in the Definition 4.1. So, we have%
\begin{equation*}
\left\langle V_{n-(i+2)},\nabla f\right\rangle =H_{i+1}^{\ast }\left\langle
V_{n},\nabla f\right\rangle
\end{equation*}%
by using (4.16) and (4.17). It follows that the equality (4.1) is true for $%
i+1$. Consequently, we get%
\begin{equation*}
\left\langle V_{n-(i+1)},\nabla f\right\rangle =H_{i}^{\ast }\left\langle
V_{n},\nabla f\right\rangle
\end{equation*}%
for all $i$ according to induction theory. This completes the proof.
\end{proof}

\begin{theorem}
Let $M^{n}$ be a $n$-dimensional Riemannian manifold with the metric $%
\left\langle ,\right\rangle $ and complete connected smooth without
boundary. Let $M^{n}$ be isometric to a Riemannian product $N\times 
\mathbb{R}
$. Let us assume that $f:M^{n}\rightarrow 
\mathbb{R}
$ be a non-trivial affine function (see main Theorem in \cite{Innami}) and $%
\alpha \left( s\right) $ be a Frenet curve of proper $n$ in $M^{n}$. If $%
\alpha $ is a $f$-eikonal $V_{n}$-slant helix curve with the axis $\nabla f$%
, then the axis of the curve $\alpha $%
\begin{equation*}
\nabla f=\left\{ H_{n-2}^{\ast }V_{1}+...+H_{1}^{\ast }V_{n-2}+V_{n}\right\}
\left\langle \nabla f,V_{n}\right\rangle \text{,}
\end{equation*}%
where $\left\{ V_{1},V_{2},...,V_{n}\right\} $ and $\left\{ H_{1}^{\ast
},...,H_{n-2}^{\ast }\right\} $ are the Frenet frame and the harmonic
curvatures of $\alpha $, respectively.
\end{theorem}

\begin{proof}
Since $\alpha $ is a $f$-eikonal $V_{n}$-slant helix curve , we can write%
\begin{equation}
\left\langle \nabla f,V_{n}\right\rangle =\text{constant.}
\end{equation}%
If we take the derivative in each part of (4.18) in the direction $V_{1}$ in 
$M^{n}$, then we have%
\begin{equation}
\left\langle \nabla _{V_{1}}\nabla f,V_{n}\right\rangle +\left\langle \nabla
f,\nabla _{V_{1}}V_{n}\right\rangle =0\text{.}
\end{equation}%
On the other hand, from Lemma 2.3 (see \cite{Sakai} ), $\nabla f$ is
parallel in $M^{n}$. That's why, $\nabla _{V_{1}}\nabla f=0$. Then, we obtain%
\begin{equation}
-k_{n-1}\left\langle \nabla f,V_{n-1}\right\rangle =0
\end{equation}%
by using (4.19) and Frenet formulas. Since $k_{n-1}$ is positive function,
(4.20) implies that%
\begin{equation*}
\left\langle \nabla f,V_{n-1}\right\rangle =0\text{.}
\end{equation*}%
Hence, we can write the axis of $\alpha $ as%
\begin{equation}
\nabla f=\lambda _{1}V_{1}+\lambda _{2}V_{2}+...+\lambda
_{n-2}V_{n-2}+\lambda _{n}V_{n}\text{.}
\end{equation}%
Moreover, from (4.21), we get%
\begin{eqnarray*}
\lambda _{1} &=&\left\langle \nabla f,V_{1}\right\rangle \\
\lambda _{2} &=&\left\langle \nabla f,V_{2}\right\rangle \\
&&. \\
&&. \\
&&. \\
\lambda _{n-2} &=&\left\langle \nabla f,V_{n-2}\right\rangle \\
\lambda _{n} &=&\left\langle \nabla f,V_{n}\right\rangle
\end{eqnarray*}%
by using Riemannian product. On the other hand, from Theorem 4.1, we know
that 
\begin{eqnarray}
\lambda _{1} &=&\left\langle \nabla f,V_{1}\right\rangle =H_{n-2}^{\ast
}\left\langle \nabla f,V_{n}\right\rangle \\
\lambda _{2} &=&\left\langle \nabla f,V_{2}\right\rangle =H_{n-3}^{\ast
}\left\langle \nabla f,V_{n}\right\rangle  \notag \\
&&.  \notag \\
&&.  \notag \\
&&.  \notag \\
\lambda _{n-2} &=&\left\langle \nabla f,V_{n-2}\right\rangle =H_{1}^{\ast
}\left\langle \nabla f,V_{n}\right\rangle  \notag \\
\lambda _{n} &=&\left\langle \nabla f,V_{n}\right\rangle  \notag
\end{eqnarray}%
Thus, it can be easily obtained the axis of the curve $\alpha $ as 
\begin{equation*}
\nabla f=\left\{ H_{n-2}^{\ast }V_{1}+...+H_{1}^{\ast }V_{n-2}+V_{n}\right\}
\left\langle \nabla f,V_{n}\right\rangle \text{,}
\end{equation*}%
by making use of the equality (4.21) and the system (4.22). This completes
the proof.
\end{proof}

\begin{theorem}
Let $M^{n}$ be a $n$-dimensional Riemannian manifold with the metric $%
\left\langle ,\right\rangle $ and complete connected smooth without
boundary. Let $M^{n}$ be isometric to a Riemannian product $N\times 
\mathbb{R}
$. Let us assume that $f:M^{n}\rightarrow 
\mathbb{R}
$ be a non-trivial affine function (see main Theorem in \cite{Innami}) and $%
\alpha \left( s\right) $ be a Frenet curve of proper $n$ in $M^{n}$. If $%
\alpha $ is a $f$-eikonal $V_{n}$-slant helix curve, then $H_{n-2}^{\ast
}\neq 0$ and $H_{1}^{\ast 2}+H_{2}^{\ast 2}+...+H_{n-2}^{\ast 2}$ is
non-zero constant, where $\left\{ H_{1}^{\ast },...,H_{n-2}^{\ast }\right\} $
is the harmonic curvatures of $\alpha $.
\end{theorem}

\begin{proof}
Let $\alpha $ be a $f$-eikonal $V_{n}$-slant helix curve and $\left\{
V_{1},...,V_{n}\right\} $ be the Frenet frame of $\alpha $. Then, from
Theorem 4.2, we know that 
\begin{equation}
\nabla f=\left\{ H_{n-2}^{\ast }V_{1}+...+H_{1}^{\ast }V_{n-2}+V_{n}\right\}
\left\langle \nabla f,V_{n}\right\rangle \text{.}
\end{equation}%
Therefore, from (4.23), we can write%
\begin{equation}
\left\langle \nabla f,\nabla f\right\rangle =\left\langle \nabla
f,V_{n}\right\rangle ^{2}\left( H_{n-2}^{\ast 2}+...+H_{1}^{\ast 2}+1\right) 
\text{.}
\end{equation}%
Moreover, by the definition of Riemannian metric, we have 
\begin{equation*}
\left\langle \nabla f,\nabla f\right\rangle =\left\Vert \nabla f\right\Vert
^{2}\text{.}
\end{equation*}%
According to this Theorem, $\alpha $ is a $f$-eikonal $V_{n}$-slant helix
curve. So, $\left\Vert \nabla f\right\Vert =$constant and $\left\langle
\nabla f,V_{n}\right\rangle =$constant along $\alpha $. Hence, from (4.24),
we obtain%
\begin{equation*}
H_{1}^{\ast 2}+H_{2}^{\ast 2}+...+H_{n-2}^{\ast 2}=\text{constant.}
\end{equation*}%
Now, we will show that $H_{n-2}^{\ast }\neq 0$ . We assume that $%
H_{n-2}^{\ast }=0$. Then, for $i=n-2$ in (4.1),%
\begin{equation}
\left\langle V_{1},\nabla f\right\rangle =H_{n-2}^{\ast 2}\left\langle
\nabla f,V_{n}\right\rangle =0\text{.}
\end{equation}%
If we take derivative in each part of (4.25) in the direction $V_{1}$ on $%
M^{n}$, then we have%
\begin{equation}
\left\langle \nabla _{V_{1}}V_{1},\nabla f\right\rangle +\left\langle
V_{1},\nabla _{V_{1}}\nabla f\right\rangle =0\text{.}
\end{equation}%
On the other hand, from Lemma 2.3 (see \cite{Sakai}), $\nabla f$ is parallel
in $M^{n}$. That's why $\nabla _{V_{1}}\nabla f=0$. Then, from (4.26), we
have $\left\langle \nabla _{V_{1}}V_{1},\nabla f\right\rangle
=k_{1}\left\langle V_{2},\nabla f\right\rangle =0$ by using the Frenet
formulas. Since $k_{1}$ is positive, $\left\langle V_{2},\nabla
f\right\rangle =0$. Now, for $i=n-3$ in (4.1),%
\begin{equation*}
\left\langle V_{2},\nabla f\right\rangle =H_{n-3}^{\ast }\left\langle
V_{n},\nabla f\right\rangle \text{.}
\end{equation*}%
And, since $\left\langle V_{2},\nabla f\right\rangle =0$, $H_{n-3}^{\ast }=0$%
. Continuing this process, we get $H_{1}^{\ast }=0$. Let us recall that $%
H_{1}^{\ast }=\frac{k_{n-1}}{k_{n-2}}$, thus we have a contradiction because
all the curvatures are nowhere zero. Consequently, $H_{n-2}^{\ast }\neq 0$.
This completes the proof.
\end{proof}

\begin{lemma}
Let $\alpha \left( s\right) $ be a Frenet curve of proper $n$ in $n$%
-dimensional Riemannian manifold $M^{n}$ and let $H_{n-2}^{\ast }\neq 0$ be
for $i=n-2$. Then, $H_{1}^{\ast 2}+H_{2}^{\ast 2}+...+H_{n-2}^{\ast 2}$ is a
nonzero constant if and only if $V_{1}\left[ H_{n-2}^{\ast }\right]
=H_{n-2}^{\ast \prime }=k_{1}H_{n-3}^{\ast }$ ,where $V_{1}$ and $\left\{
H_{1}^{\ast },...,H_{n-2}^{\ast }\right\} $ are the unit vector tangent
vector field and the Harmonic curvatures of $\alpha $, respectively.
\end{lemma}

\begin{proof}
First, we assume that $H_{1}^{\ast 2}+H_{2}^{\ast 2}+...+H_{n-2}^{\ast 2}$
is a nonzero constant . Consider the functions%
\begin{equation*}
H_{i}^{\ast }=\left\{ k_{n-i}H_{i-2}^{\ast }-H_{i-1}^{\ast \prime }\right\} 
\frac{1}{k_{n-(i+1)}}
\end{equation*}%
for $3\leq i\leq n-2$. So, from the equality, we can write%
\begin{equation}
k_{n-(i+1)}H_{i}^{\ast }=k_{n-i}H_{i-2}^{\ast }-H_{i-1}^{\ast \prime }\text{%
, }3\leq i\leq n-2\text{.}
\end{equation}%
Hence, in (4.27), if we take $i+1$ instead of $i$, we get%
\begin{equation}
H_{i}^{\ast \prime }=k_{n-(i+1)}H_{i-1}^{\ast }-k_{n-(i+2)}H_{i+1}^{\ast
},2\leq i\leq n-3
\end{equation}%
together with%
\begin{equation}
H_{1}^{\ast \prime }=-k_{n-3}H_{2}^{\ast }\text{.}
\end{equation}%
On the other hand, since $H_{1}^{\ast 2}+H_{2}^{\ast 2}+...+H_{n-2}^{\ast 2}$
is constant, we have%
\begin{equation*}
H_{1}^{\ast }H_{1}^{\ast \prime }+H_{2}^{\ast }H_{2}^{\ast \prime
}+...+H_{n-2}^{\ast }H_{n-2}^{\ast \prime }=0
\end{equation*}%
and so, 
\begin{equation}
H_{n-2}^{\ast }H_{n-2}^{\ast \prime }=-H_{1}^{\ast }H_{1}^{\ast \prime
}-H_{2}^{\ast }H_{2}^{\ast \prime }-...-H_{n-3}^{\ast }H_{n-3}^{\ast \prime }%
\text{.}
\end{equation}%
By using (4.28) and (4.29), we obtain%
\begin{equation}
H_{1}^{\ast }H_{1}^{\ast \prime }=-k_{n-3}H_{1}^{\ast }H_{2}^{\ast }
\end{equation}%
and%
\begin{equation}
H_{i}^{\ast }H_{i}^{\ast \prime }=k_{n-(i+1)}H_{i-1}^{\ast }H_{i}^{\ast
}-k_{n-(i+2)}H_{i}^{\ast }H_{i+1}^{\ast }\text{, }2\leq i\leq n-3\text{.}
\end{equation}%
Therefore, by using (4.30), (4.31) and (4.32), a algebraic calculus shows
that%
\begin{equation*}
H_{n-2}^{\ast }H_{n-2}^{\ast \prime }=k_{1}H_{n-3}^{\ast }H_{n-2}^{\ast }%
\text{.}
\end{equation*}%
Since $H_{n-2}^{\ast }\neq 0$, we get the relation $H_{n-2}^{\ast \prime
}=k_{1}H_{n-3}^{\ast }$.

Conversely, we assume that%
\begin{equation}
H_{n-2}^{\ast \prime }=k_{1}H_{n-3}^{\ast }.
\end{equation}%
By using (4.33) and $H_{n-2}^{\ast }\neq 0$, we can write%
\begin{equation*}
H_{n-2}^{\ast }H_{n-2}^{\ast \prime }=k_{1}H_{n-2}^{\ast }H_{n-3}^{\ast }
\end{equation*}%
From (4.32), we have 
\begin{eqnarray*}
\text{for }i &=&n-3\text{, \ \ \ \ \ }H_{n-3}^{\ast }H_{n-3}^{\ast \prime
}=k_{2}H_{n-4}^{\ast }H_{n-3}^{\ast }-k_{1}H_{n-3}^{\ast }H_{n-2}^{\ast } \\
\text{for }i &=&n-4\text{, \ \ \ \ \ }H_{n-4}^{\ast }H_{n-4}^{\ast \prime
}=k_{3}H_{n-5}^{\ast }H_{n-4}^{\ast }-k_{2}H_{n-4}^{\ast }H_{n-3}^{\ast } \\
\text{for }i &=&n-5\text{, \ \ \ \ \ }H_{n-5}^{\ast }H_{n-5}^{\ast \prime
}=k_{4}H_{n-6}^{\ast }H_{n-5}^{\ast }-k_{3}H_{n-5}^{\ast }H_{n-4}^{\ast } \\
&&\cdot \\
&&\cdot \\
&&\cdot \\
\text{for }i &=&2\text{, \ \ \ \ \ \ \ \ \ \ }H_{2}^{\ast }H_{2}^{\ast
\prime }=k_{n-3}H_{1}^{\ast }H_{2}^{\ast }-k_{n-4}H_{2}^{\ast }H_{3}^{\ast }
\end{eqnarray*}%
and from (4.31), we have%
\begin{equation*}
H_{1}^{\ast }H_{1}^{\ast \prime }=-k_{n-3}H_{1}^{\ast }H_{2}^{\ast }\text{.}
\end{equation*}%
So, an algebraic calculus shows that%
\begin{equation}
H_{1}^{\ast }H_{1}^{\ast \prime }+H_{2}^{\ast }H_{2}^{\ast \prime }+...+%
\text{\ }H_{n-5}^{\ast }H_{n-5}^{\ast \prime }+H_{n-4}^{\ast }H_{n-4}^{\ast
\prime }+H_{n-3}^{\ast }H_{n-3}^{\ast \prime }+H_{n-2}^{\ast }H_{n-2}^{\ast
\prime }=0\text{.}
\end{equation}%
And, by integrating (4.34), we can easily say that%
\begin{equation*}
H_{1}^{\ast 2}+H_{2}^{\ast 2}+...+H_{n-2}^{\ast 2}
\end{equation*}%
is a nonzero constant. This completes the proof.
\end{proof}

\begin{corollary}
Let $M^{n}$ be a $n$-dimensional Riemannian manifold with the metric $%
\left\langle ,\right\rangle $ and complete connected smooth without
boundary. Let $M^{n}$ be isometric to a Riemannian product $N\times 
\mathbb{R}
$. Let us assume that $f:M^{n}\rightarrow 
\mathbb{R}
$ be a non-trivial affine function (see main Theorem in \cite{Innami}) and $%
\alpha \left( s\right) $ be a Frenet curve of proper $n$ in $M^{n}$. If $%
\alpha $ is a $f$-eikonal $V_{n}$-slant helix curve, then $H_{n-2}^{\ast
\prime }=k_{1}H_{n-3}^{\ast }$.
\end{corollary}

\begin{proof}
It is obvious by using Theorem 4.3 and Lemma 4.1.
\end{proof}

\end{document}